\documentclass [a4paper] {article}

\usepackage[utf8]{inputenc}

\usepackage[french]{babel}

\usepackage{amssymb,amsmath, amsfonts}

\usepackage [dvips] {graphicx,color}

\usepackage{amsthm}

\usepackage{hyperref}

\usepackage{MnSymbol}

\newtheorem{prop} {Proposition} 
 
\newtheorem{thm} [prop] {Théorème}

\theoremstyle{definition}
\newtheorem{df}{Définition} 
\newtheorem*{df*}{Définition}

\theoremstyle{remark}
\newtheorem{rmq}{Remarque} 
\newtheorem{example}{Exemple} 
\newtheorem{exm}[example]{Exemple}

\author{Stéphane \textsc{Dugowson}
\footnote{s.dugowson@gmail.com}
}

\title {Définition du topos d'un espace connectif}

\begin{document}

\maketitle

\paragraph{Résumé.} Cette courte note définit le topos de Grothendieck associé à un espace connectif.

\subparagraph{\emph{Mots clés.}} Topos de Grothendieck. Espaces connectifs.

\paragraph{Abstract.} \textsc{The Topos Of A Connectivity Space ---} In this short note, a topos --- called \emph{the topos of the connectivity space} --- is associated with every such space.

\subparagraph{\emph{Key words.}} Grothendieck Topos. Connectivity spaces. 

\paragraph{MSC 2010 :} 18B25, 54A05.\\

L'objet de cette note est de définir le \emph{topos associé à un espace connectif}. Pour les notations et les définitions, on se reportera à \cite{Dugowson:201012}. 
Soit $X=(X,\mathcal{K})$ un espace connectif, que nous ne supposons pas nécessairement intègre.

L'ensemble ordonné $(\mathcal{K},\subset)$ s'identifie à une catégorie que nous noterons encore $\mathcal{K}$. Bien entendu, on peut considérer le topos des préfaisceaux sur cet ensemble ordonné, mais ce n'est pas celui-là (en général) que nous appellerons le topos associé à $\mathcal{K}$. 
Pour tout connexe $K\in\mathcal{K}$, on note $\mathbf{c}(K)$ l'ensemble des cribles sur $K$. Tout crible sur $K$ sera identifié à l'ensemble de parties connexes de $K$ que sont les domaines des flèches du crible en question.

\begin{rmq} Pour tout connexe $A\in\mathcal{K}$, la structure connective induite sur $A$, que nous noterons $\mathcal{K}_{\vert A}$ et qui est donnée par
\[\mathcal{K}_{\vert A}=\mathcal{K} \cap \mathcal{P}A \]
constitue le crible maximal sur $A$ : \[c_{\max}(A)=\mathcal{K}_{\vert A}.\]
\end{rmq}

Pour tout $A\in\mathcal{K}$, on pose
\[
J(A)=\{
c\in \mathbf{c}(A), 
[c]_0 
=\mathcal{K}_{\vert A}
\},
\]
où $[c]_0$ désigne la structure connective\footnote{Non nécessairement intègre.} engendrée par l'ensemble $c\subset \mathcal{K}$. Anticipant l'annonce du théorème \ref{thm topo G sur K} ci-dessous, pour tout connexe  $A\in \mathcal{K}$, les éléments de $J(A)$ seront appelés les \emph{cribles couvrants $A$} (ou les cribles qui recouvrent $A$).

\begin{rmq}
Les points  ne couvrent qu'eux-mêmes : les connexes non réduits à un point ne sont jamais engendrés par les singletons seuls.
\end{rmq}

La proposition suivante, dans laquelle $\mathcal{I}_\mathcal{K}$ désigne l'ensemble des connexes irréductibles, découle immédiatement des définitions.

\begin{prop}
Une partie connexe $K\in\mathcal{K}$ de $X$ est irréductible si et seulement si $J(K)$ est un singleton, seul le crible maximal étant alors couvrant :
\[(K \in\mathcal{I}_\mathcal{K})\Leftrightarrow J(K)=\{c_{\max}(K)\}.\]
\end{prop}

\begin{rmq} Pour tout connexe $K\subset X$, il y a deux cribles distincts d'union la partie vide --- toujours connexe --- de $K$  : le crible vide $\emptyset$  et le crible $\{\emptyset\}$ dont le seul élément est l'injection canonique $\emptyset \hookrightarrow K$. Si $K=\emptyset$, ces deux cribles couvrent $\emptyset$ : ainsi,  $J(\emptyset)$ a deux éléments --- le crible maximal et le crible vide --- et la partie vide $\emptyset\subset X$ n'est donc pas irréductible.
\end{rmq}

\begin{thm} \label{thm topo G sur K}
$J$ constitue une topologie de Grothendieck sur la catégorie
$\mathcal{K}$. 
\end{thm}

\begin{proof} En effet,
\begin{itemize}
\item Pour tout  $A\in\mathcal{K}$, le crible maximal $c_{\max}(A)$ recouvre trivialement $A$,
\item La restriction à un connexe $B\subset A$ d'un crible couvrant $A$  est  couvrante pour $B$, puisque $\mathcal{K}_B$ est engendré par des éléments du crible, qui sont eux-mêmes nécessairement inclus dans $B$,
\item Toute couverture de chacun des connexes d'une famille couvrant $K$  détermine une famille couvrante de $K$ puisque la structure connective engendrée par celle-ci, contenant nécessairement ceux qui appartiennent à celle-là, contient également la structure que ces derniers engendrent.
\end{itemize}
\end{proof}

\begin{df}
Le \emph{site de $\mathcal{K}$} est le site $(\mathcal{K},J)$.  
Le topos $\mathcal{T}_{(X,\mathcal{K})}=\mathcal{T}_\mathcal{K}$ 
associé à un espace connectif 
$(X,\mathcal{K})$ est le topos des faisceaux d'ensembles sur le site
 $(\mathcal{K},J_\mathcal{K})$.
\end{df}

\begin{exm} On prend sur $X=\{a,b,c,d,e\}$ la structure connective \[\mathcal{K}=\{\emptyset,\{a\}, \{b\},\{c\},\{d\},\{e\}, \{a,b\}, \{b,c,d\}, \{a,b,c,d\}, X\}.\] Tous les $J(K)$ sont réduits au crible maximal sur $K$, à l'exception du vide $\emptyset$ et de $\{a,b,c,d\}$ qui est déjà couvert par le crible $\{\emptyset,\{a\}, \{b\},\{c\},\{d\}, \{a,b\}, \{b,c,d\}\}.$
\end{exm}

\begin{exm} Soit $\mathcal{K}_\mathbf{R}$ la structure connective usuelle sur $\mathbf{R}$, c'est-à-dire l'ensemble des intervalles. Pour tout intervalle $I$ et tout $\epsilon> 0$,  l'ensemble des sous-intervalles de $I$ de longueur $<\epsilon$ est un crible couvrant. 
\end{exm}

\paragraph{Conclusion}
Naturellement, parmi les nombreuses questions qui se posent d'emblée, se pose en particulier celle de savoir s'il est possible d'associer un morphisme géométrique à un morphisme connectif.

\bibliographystyle{plain}




\end{document}